\documentclass[10pt,a4paper]{article}
\usepackage[utf8]{inputenc}
\usepackage{amsmath}
\usepackage{amsfonts}
\usepackage{amssymb}
\usepackage{amsthm}
\usepackage{graphicx}
\usepackage{verbatim}
\usepackage{caption}
\usepackage{subcaption}
\usepackage{color}

\newtheoremstyle{firststyle}  % name of the style to be used
  {5mm}       % measure of space to leave above the theorem. E.g.: 3pt
  {3mm}       % measure of space to leave below the theorem. E.g.: 3pt
  {\itshape}   % name of font to use in the body of the theorem
  {}        % measure of space to indent
  {\bfseries}  % name of head font
  {}   % punctuation between head and body
  {3mm}       % space after theorem head
  {}           % Manually specify head

\theoremstyle{firststyle}
\newtheorem{theorem}{Theorem}[section]
\newtheorem{lemma}[theorem]{Lemma}
\newtheorem{proposition}[theorem]{Proposition}

\newtheoremstyle{secondstyle}  % name of the style to be used
  {3mm}       % measure of space to leave above the theorem. E.g.: 3pt
  {3mm}       % measure of space to leave below the theorem. E.g.: 3pt
  {\normalfont}   % name of font to use in the body of the theorem
  {}        % measure of space to indent
  {\bfseries}  % name of head font
  {\quad}   % punctuation between head and body
  {1mm}       % space after theorem head
  {}           % Manually specify head

\theoremstyle{secondstyle}

\newtheorem{remark}{Remark}
\newtheorem{assumption}{Assumption}

\usepackage{booktabs} % making rules for the table
\title{Error analysis for global minima of semilinear optimal control problems} 
\author{Ahmad Ahmad Ali\footnote{Schwerpunkt Optimierung und Approximation, 
		Universit\"at Hamburg, Bundesstra{\ss}e 55, 20146 Hamburg, Germany.}, Klaus Deckelnick\footnote{Institut f\"ur Analysis und Numerik,
		Otto--von--Guericke--Universit\"at Magdeburg, Universit\"atsplatz 2,
		39106 Magdeburg, Germany} \& Michael Hinze\footnote{Schwerpunkt Optimierung und Approximation, 
		Universit\"at Hamburg, Bundesstra{\ss}e 55, 20146 Hamburg, Germany.}}

\date{}

\begin{document}
\maketitle
\begin{center}
 {\it  Dedicated to Eduardo Casas on the occasion of his 60th birthday.}
\end{center}

{\small {\bf Abstract:} In \cite{ali2016global} we consider an optimal control problem subject to a semilinear elliptic PDE together with its variational discretization, where we provide a condition which allows to decide whether a solution of the necessary first order conditions is a global minimum. This condition can be explicitly evaluated at the discrete level. Furthermore, we prove that if the above condition holds uniformly with respect to the discretization parameter the sequence of discrete solutions converges to a global solution of the corresponding limit problem. With the present work we complement our investigations of \cite{ali2016global} in that we prove an error estimate for those discrete global solutions. Numerical experiments confirm our analytical findings.}
\\[2mm]

\section{Introduction}
In this work we are concerned with  the error analysis of a variational discretization  of the control problem
\[
(\mathbb{P}) \quad
\min_{u \in U_{ad}} J(u):=\frac{1}{2} \Vert y-y_0 \Vert_{L^2(\Omega)}^2  + \frac{\alpha}{2} \Vert u \Vert_{L^2(\Omega)} ^2
\]
subject to 
\begin{eqnarray}
-\Delta y +\phi(y) &=u  \quad \mbox{ in }  \Omega, \label{semilinear}\\
y &=0  \quad \mbox{ on } \partial\Omega,  \label{bc}  
\end{eqnarray}
and the pointwise constraints
\begin{alignat*}{2}
& u_a \leq u(x) \leq u_b      \qquad  &&\mbox{for  a.e. }  x \in \Omega,\\
& y_a(x) \leq y(x)\leq y_b(x)  \qquad &&\forall \, x \in K \subset \Omega,
\end{alignat*}
where the precise assumptions on the data of the problem will be given in Section~\ref{sec:problem setting}. In \cite{ali2016global} the authors considered  the same class of problems and established a sufficient condition for the global minima of $(\mathbb{P})$ assuming particular types of growth conditions for the nonlinearity $\phi$. The same result was established for the variational discrete counterpart of $(\mathbb{P})$, and it was shown that a sequence of the computed discrete global minima  converges to a global minimum of the continuous control problem but without discussing the corresponding rate of convergence. Hence, our aim in this study is to investigate this convergence rate.

The organization of the paper is as follows: in \S ~\ref{sec:problem setting} we formulate the control problem and give the exact assumptions on the data. In \S~\ref{sec:state} and in \S~\ref{sec:control problem} we review the results concerning the state equation and the control problem $(\mathbb{P})$, respectively. The variational discretization of $(\mathbb{P})$ is considered in \S~\ref{sec:variational} while \S~\ref{sec:error} is devoted to the error analysis. Finally, in \S~\ref{sec:numerics} we verify our theoretical findings by a numerical example.

Before starting, we give a short list of literature considering the problem $(\mathbb{P})$. For a broad overview, we refer the reader to the references of the respective citations. In \cite{casas1993boundary} the problem  $(\mathbb{P})$ is studied when the controls are of boundary type, and the necessary first order conditions are established. Compare \cite{casas1986control} where the function $\phi$ is linear, and \cite{casas1993optimal} where the pointwise constraints are imposed on the gradient of the state.

The regularity of the optimal controls  of $(\mathbb{P})$ and their associated multipliers are investigated in \cite{casas2010recent} and  \cite{casas2014new}, where also the sufficient second order conditions are discussed. Compare  \cite{casas2002second,casas2008necessary,casas2008sufficient} for second order conditions when the set $K$ contains  finitely/infinitely many points, and 
 \cite{casas2015second} for the role of those conditions in PDE constrained control problems.

Finite element discretization of problem $(\mathbb{P})$ under more general setting  is studied in \cite{casas2002uniform}, and in \cite{hinze2012stability} where a wider class of perturbations are considered. The convergence of the discrete solutions to the continuous solutions is verified there but without rates. However, when the set $K$ contains finitely many points, convergence rates are  established in  \cite{merino2010error} for finite dimensional controls, and in \cite{casas2002error} for control functions. Only in \cite{neitzel2015finite} error analysis is studied for general pointwise state constraints in $K$. There, Pfefferer at al. prove an error estimate for discrete solutions in the vicinity of a local solution which satisfies a quadratic growth condition. Error analysis for linear-quadratic control problems  can be found in e.g. \cite{casas2014new}, \cite{deckelnick2007convergence} and \cite{meyer2008error}. A detailed discussion of discretization concepts and error analysis in PDE-constrained control problems can be found in \cite{hinze2012discretization,hinze2010discrete} and\cite[Chapter~3]{pinnau2008optimization}.

\section{Problem Setting and discretization}

\subsection{Assumptions}
\label{sec:problem setting}

\begin{itemize}
\item $\Omega \subset \mathbb{R}^2$ is a bounded, convex and polygonal domain.
\item $K$ is a (possibly empty) compact subset of $\Omega$.
\item $u_a \in \mathbb{R} \cup \{-\infty\}$ and $u_b \in \mathbb{R} \cup \{\infty\}$ with $u_a \leq u_b$.
\item $y_a,y_b \in C_0(\Omega)\cap W^{2,\infty}(\Omega)$ are given functions that satisfy $y_a(x)<y_b(x)$, $x \in K$.
\item $y_0 \in L^2(\Omega)$ and $\alpha >0$ are given.
\item $\phi: \mathbb{R} \to \mathbb{R}$ is of class $C^2$ and monotonically increasing. 
\item There exist $r>1$ and $M \geq 0$ such that
\begin{align} 
|\phi''(s)| \leq M \phi'(s)^{\frac{1}{r}} \quad \mbox{ for all } s \in \mathbb{R}, \label{assumption:D2phi}
\end{align}
where $\phi'$ and $\phi''$ denote the first and second derivative of $\phi$, respectively.
\end{itemize}

\subsection{The State Equation}
\label{sec:state}

Recall that a function $y \in H^1_0(\Omega)$ is called a weak solution of \eqref{semilinear}, \eqref{bc} if
\begin{align}
\int_\Omega \nabla y   \cdot  \nabla v +  \phi(y) v \, dx = \int_\Omega u v \,dx \quad \forall \, v \in H^1_0(\Omega). \label{WPDE:semilinear state}
\end{align}

\begin{theorem}
\label{Thm:H2 regularity state}
For every  $u \in L^2(\Omega)$ the boundary value problem \eqref{semilinear}, \eqref{bc} admits a unique weak solution $y \in H^1_0(\Omega)\cap H^2(\Omega)$. Moreover, there exists $c>0$ such that 
\begin{align}
\| y\|_{H^2(\Omega)} \leq c \bigl( 1+ \| u\|_{L^2(\Omega)} \bigr). \label{estimate:H2 state}
\end{align}
\end{theorem}
\begin{proof}
The existence and uniqueness of the solution $y$ in $H^1_0(\Omega)$ follows from the monotone operator theorem. Using the method of Stampacchia one can show, in addition, that $y \in L^\infty(\Omega)$. Utilizing the boundedness of $y$ and the properties of the nonlinearity $\phi$, one can show   $y \in H^2(\Omega)$ and the estimate \eqref{estimate:H2 state} using the regularity results from \cite[Chapter~4]{grisvard2011elliptic}. For a detailed proof compare for instance \cite{casas1993boundary}.
\end{proof}

In the light of Theorem~\ref{Thm:H2 regularity state}, we introduce the control--to--state operator 
\begin{align}
\mathcal{G}: L^2(\Omega) \to   H^1_0(\Omega) \cap H^2(\Omega) \label{solution operator}
\end{align}
such that $y:=\mathcal{G}(u)$ is the solution to \eqref{WPDE:semilinear state} for a given $u \in L^2(\Omega)$.

\begin{lemma}
\label{Thm:lipschitz map}
Let $\mathcal{G}$ be the mapping introduced in  \eqref{solution operator}. Then there exists $c>0$ depending only on $\Omega$ such that
\begin{align*}
\|\mathcal{G}(u)-\mathcal{G}(v)\|_{L^2(\Omega)} \leq c \|u-v\|_{L^2(\Omega)} \qquad \forall \, u,v \in L^2(\Omega).
\end{align*}
\end{lemma}
\begin{proof}
Given $u,v \in L^2(\Omega)$ let $y_u:=\mathcal{G}(u)$ and $y_v:=\mathcal{G}(v)$. Using Poincar\'{e}'s inequality, the monotonicity of $\phi$ and \eqref{WPDE:semilinear state} we have
\begin{eqnarray*}
\lefteqn{
\|y_u - y_v\|^2_{L^2(\Omega)}  \leq c\int_{\Omega} |\nabla(y_u-y_v)|^2 \, dx } \\
& \leq &  \int_{\Omega} |\nabla(y_u-y_v)|^2  + [\phi(y_u)-\phi(y_v)](y_u-y_v) \, dx \\
& = &  \int_{\Omega} (u-v)(y_u-y_v) \, dx  \leq \|u-v\|_{L^2(\Omega)} \|y_u-y_v\|_{L^2(\Omega)},
\end{eqnarray*}
which implies the result.
\end{proof}

\begin{lemma}
\label{Thm:H2 lipschitz}
Let $\mathcal{G}$ be the mapping introduced in \eqref{solution operator}. Then for any $m>0$ there exists $L(m)>0$ such that
\begin{align*}
\|\mathcal{G}(u)-\mathcal{G}(v)\|_{H^2(\Omega)} \leq L(m) \|u-v\|_{L^2(\Omega)} 
\end{align*}
for all $u$, $v \in L^2(\Omega)$ with $\|u\|_{L^2(\Omega)}$, $\|v\|_{L^2(\Omega)} \leq m$.
\end{lemma}

\begin{proof}
Defining again $y_u:=\mathcal{G}(u), \, y_v:=\mathcal{G}(v)$ we infer from Theorem~\ref{Thm:H2 regularity state} and the continuous embedding $H^2(\Omega) \hookrightarrow C(\bar\Omega)$ 
that $\|y_u\|_{L^\infty(\Omega)}$, $\|y_v\|_{L^\infty(\Omega)} \leq c_m$ for some $c_m >0$ depending on $m$. Clearly,
$y_u-y_v$ belongs to $H^1_0(\Omega) \cap H^2(\Omega)$ and satisfies 
\begin{align*}
-\Delta (y_u - y_v) = (u_1-u_2)-[\phi(y_u)-\phi(y_v)] \quad \mbox{ in }\Omega.
\end{align*}
Using a standard a--priori estimate, the Lipschitz continuity of $\phi$ on bounded sets and Lemma \ref{Thm:lipschitz map} we infer that 
\begin{align*}
\|y_1 - y_2\|_{H^2(\Omega)} & \leq c \big(\|u - v\|_{L^2(\Omega)}+\|\phi(y_u)-\phi(y_v)\|_{L^2(\Omega)} \big) \\
& \leq L(m) \big(\|u - v\|_{L^2(\Omega)}+\|y_u - y_v\|_{L^2(\Omega)} \big) \\
& \leq L(m) \|u - v\|_{L^2(\Omega)},
\end{align*}
where $L(m)$ is a constant depending on $m$. This completes the proof.
\end{proof}

\subsection{The Optimal Control Problem $(\mathbb{P})$}
\label{sec:control problem}

Using the control-to-state operator $\mathcal{G}$  defined in \eqref{solution operator}, the reduced form of  our optimal control problem reads
\[
(\mathbb{P}) \quad
\begin{array}{l}
\min_{u \in U_{ad}} J(u):=\frac{1}{2} \Vert y-y_0 \Vert_{L^2(\Omega)}^2  + \dfrac{\alpha}{2} \Vert u \Vert_{L^2(\Omega)} ^2 \\
\mbox{subject to } y=\mathcal{G}(u) \mbox{ and } y|_K \in Y_{ad},
\end{array}
\]
where 
\begin{align*}
 U_{ad} &:=\{ v \in L^2(\Omega) : u_a \leq v(x) \leq u_b \mbox{ a.e. in } \Omega \},\\
Y_{ad} &:=\{ z \in C(K): y_a(x)\leq z(x)\leq y_b(x) \mbox{ for all }  x \in K \}.
\end{align*}
It is well--known that $(\mathbb{P})$ admits at least one solution provided that a feasible point exists (compare \cite{casas1993boundary}).  Moreover, if a solution of $(\mathbb{P})$ satisfies some constraint qualification, then one can guarantee the existence of a multiplier associated with the pointwise state constraints and the necessary first order conditions can be established. 
A typical constraint qualification for a local solution $\bar u$ of problem $(\mathbb{P})$ is the linearized Slater condition which reads: there exist $u_0 \in U_{ad}$ and $\delta >0$ such that
\begin{align}
 y_a(x) +  \delta \leq \mathcal{G}(\bar u)(x)+ \mathcal{G}'(\bar u)(u_0 - \bar u)(x) \leq y_b(x)-\delta \qquad \forall \, x \in K. \label{slater condition}
\end{align}
The next result is a consequence of \cite[Theorem~5.2]{casas1993boundary}.
\begin{theorem}
Let $\bar u \in U_{ad}$ be a local solution of problem $(\mathbb{P})$ satisfying \eqref{slater condition}. Then there exist  $\bar p \in W^{1,s}_0(\Omega)$ for $1<s<2$ and a regular Borel measure $\bar \mu \in \mathcal{M}(K)$ such that with $\bar y \in H^1_0(\Omega)\cap H^2(\Omega)$ there holds
\begin{align}
&\int_\Omega \nabla \bar y \cdot \nabla v + \phi(\bar y)v \, dx = \int_\Omega \bar u v \, dx \quad \forall \, v \in H^1_0(\Omega),  \qquad \bar y_{|K} \in Y_{ad}, \label{oc:state} \\
&\int_\Omega  \bar p (- \Delta v)  + \phi'(\bar y)\bar p v \, dx  
   \nonumber\\
& \enspace =\int_\Omega (\bar y -y_0) v \, dx   +   \int_K v \, d\bar\mu  \quad \forall \,v \in H^1_0(\Omega) \cap H^2(\Omega), \label{oc:adjoint} \\
&\int_\Omega (\bar p + \alpha \bar u)(u-\bar u) \, dx  \geq  0 \qquad \forall \, u \in U_{ad}, \label{oc:VI control}\\
&\int_K (z-\bar y)\, d\bar\mu  \leq  0 \qquad \forall \, z \in Y_{ad}. \label{oc:VI measure}
\end{align}
\end{theorem}

\noindent
Note that  in view of  \eqref{oc:VI control} $\bar{u}$ is the $L^2$--projection of $-\frac{1}{\alpha} \bar p$ onto $U_{ad}$ so that
\begin{align*}
\bar u(x) = \min \big(\max \big (u_a,-\tfrac{1}{\alpha}\bar p(x) \big),u_b \big) \quad \forall \, x \in \Omega,
\end{align*}
Since $\bar p \in  W^{1,s}(\Omega)$ for  $1<s<2$  it follows from \cite[Corollary~A.6]{kinderlehrer1980introduction} that $\bar u \in W^{1,s}(\Omega)$ for  $1<s<2$ as well. Furthermore, 
it is well known that the multiplier $\bar \mu$ associated with the pointwise state constraints is concentrated at the points in $K$ where the state constraints are active. We state this more precisely in the next proposition whose  proof can be found in \cite{casas2014new}. Compare also the proof in \cite{casas1986control} when the bounds $y_a$, $y_b$ are constant functions.
\begin{proposition}
\label{Thm:supp measure}
Let $\bar \mu \in \mathcal{M}(K)$ and $\bar y \in C_0(\Omega)$ satisfy \eqref{oc:VI measure}. Then there holds
\begin{align*}
&\mbox{supp} (\bar \mu_b) \subset \{x \in K: \bar y(x)=y_b(x) \},\\
&\mbox{supp} (\bar \mu_a) \subset \{x \in K: \bar y(x)=y_a(x) \},
\end{align*}
where $\bar \mu = \bar \mu_b - \bar \mu_a$  with  $\bar \mu_b,\bar \mu_a \geq 0$ is the Jordan decomposition of $\bar \mu$. 
\end{proposition}

We note that the problem $(\mathbb{P})$ is in general nonconvex since the state equation is not linear. In other words, the problem $(\mathbb{P})$ can have several solutions. A decision of which of these solution is a global minimum proves difficult in general. However, it is shown in \cite{ali2016global} that if the nonlinearity $\phi$ of the state equation enjoys certain growth conditions, namely \eqref{assumption:D2phi}, then one can establish a condition that helps to decide if a given point satisfying the first order conditions is a global minimum. 
We state this condition of global optimality in the next result, but before that we first need to introduce the following constant:
\begin{align}
\label{eta def}
\displaystyle
\eta(\alpha,r):= \alpha^{\frac{\rho}{2}}  C^{\frac{2-2r}{r}}_q  M^{-1} \bigg(\frac{r-1}{2r-1} \bigg)^{\frac{1-r}{r}} 
q^{1/q} r^{1/r} \rho^{\rho/2} (2 - \rho)^{\frac{\rho}{2}-1}.
\end{align}  
Here, $q:=\tfrac{3r-2}{r-1}, \, \rho:=\frac{r+q}{rq}$, while $M$ and $r$ appear in  \eqref{assumption:D2phi}.  Furthermore, $C_q$ is  an upper bound on the 
optimal constant in the Gagliardo-Nirenberg inequality 
\[
\Vert f \Vert_{L^q} \leq C \Vert f \Vert_{L^2}^{\frac{2}{q}} \Vert \nabla f \Vert_{L^2}^{\frac{q-2}{q}} \qquad \forall f \in H^1(\mathbb{R}^2).
\]
For sharp upper bounds for the constant $C$, see for instance \cite[Theorem 7.3]{ali2016global}.
\begin{theorem}
\label{Thm:global minima}
Suppose that $\bar u \in U_{ad}$, $\bar y \in H^2(\Omega) \cap H^1_0(\Omega)$, $\bar p \in W^{1,s}_0(\Omega)$ ($1<s<2$), $\bar \mu \in \mathcal M(K)$ is a solution of 
\eqref{oc:state}--\eqref{oc:VI measure}.
If
\begin{align}
\label{18strich}
\|\bar p \|_{L^q(\Omega)} \leq \eta(\alpha,r), 
\end{align}
then $\bar u$ is a global minimum for Problem~$(\mathbb{P})$. If the inequality (\ref{18strich}) is strict, then
$\bar u$ is the unique global minimum. 
\end{theorem}

\subsection{Variational Discretization}
\label{sec:variational}

Let $\mathcal T_h$ be an admissible triangulation of the polygonal domain $\Omega \subset \mathbb{R}^2$ with 
\begin{displaymath}
\overline{\Omega} = \bigcup_{T \in \mathcal T_h} \overline{T}.
\end{displaymath}
Here $h:=\max_{T \in \mathcal{T}_h} \mbox{diam}(T)$ is the maximum mesh size, while $\mbox{diam}(T)$ stands for the diameter of the triangle $T$. We
introduce the following spaces of linear finite elements:
\begin{align*}
X_{h} &:= \{ v_h \in C(\bar\Omega) : v_h|_T \mbox{ is a linear polynomial on each }  T \in \mathcal{T}_h\}, \\
X_{h0} &:= \{ v_h \in X_h : {v_h}_{ | \partial\Omega}=0   \}.
\end{align*}
The Lagrange interpolation operator $I_h$ is defined by
\begin{align*}
I_h:C(\bar\Omega) \to X_h, \qquad I_h y:= \sum_{i=1}^n y(x_i) \phi_i,
\end{align*}
where $\{x_1,\ldots,x_n\}$ denote the nodes in the triangulation  $\mathcal{T}_h$ and $\{\phi_1, \ldots, \phi_n\}$  are the basis functions of the space $X_h$ which satisfy $\phi_i(x_j)=\delta_{ij}$.

\begin{comment}
Finally, we state the next lemma which requires the following assumption on $\Omega$ and its sequence of triangulations $\{\mathcal{T}_h\}_{0<h\leq h_0}$. 
\begin{assumption}
\label{assumption:1}
There is a convex polygonal domain $\tilde \Omega$ containing $\Omega$, that is, $\Omega \subset \tilde \Omega$, such that for $h_0$ small enough each  $ \mathcal{T}_h \in \{ \mathcal{T}_h\}_{0<h \leq h_0}$ can be extended to a triangulation $\tilde{\mathcal{T}}_h$ of $\tilde \Omega$ such that the  sequence  $\{ \tilde{\mathcal{T}}_h\}_{0<h \leq h_0}$ is quasi-uniform  with the same quasi-uniformity  constant $\gamma$ of $ \{ \mathcal{T}_h\}_{0<h \leq h_0}$.
\end{assumption}

\begin{lemma}
\label{Thm:ritz projection}
Suppose that Assumption~\ref{assumption:1} holds. Let $y \in H_0^1(\Omega)\cap C(\bar\Omega)$ be given and let $y_h \in X_{h0}$ be the unique function satisfying  
\begin{align*}
\int_\Omega \nabla y_h \cdot \nabla v_h \,dx = \int_\Omega \nabla y \cdot \nabla v_h \,dx \quad \forall \, v_h \in X_{h0}.
\end{align*}
Then there holds
\begin{align}
\|y - y_h\|_{L^\infty(\Omega)} \leq c |\ln h| \inf_{\chi \in X_{h0} }\|y - \chi\|_{L^\infty(\Omega) }, \label{estimate_a}
\end{align}
for some $c>0$ independent of $h$.
\end{lemma}
\begin{proof}
 See \cite[Theorem~2]{schatz1980weak}.
\end{proof}
\end{comment}

The finite element discretization of  \eqref{WPDE:semilinear state} reads: for a given $u \in L^2(\Omega)$, find $y_h \in X_{h0}$ such that 
\begin{align}
\label{WPDE:semilinear state h}
\int_\Omega \nabla y_h   \cdot  \nabla v_h +  \phi(y_h) v_h \, dx = \int_\Omega u v_h \,dx \quad \forall \, v_h \in X_{h0}.
\end{align}
Using the monotonicity of $\phi$ and the Brouwer fixed-point theorem one can show that \eqref{WPDE:semilinear state h} admits a unique solution  $y_h \in X_{h0}$. Hence, analogously to \eqref{solution operator}, we introduce
the discrete control--to--state operator
\begin{align}
\label{solution operator h}
\mathcal{G}_h:L^2(\Omega) \to X_{h0}
\end{align}
such that $y_h:=\mathcal{G}_h(u)$ is the solution of \eqref{WPDE:semilinear state h}.

\begin{comment}
\begin{remark}
\label{remark:a}
The advantage of postulating Assumption~\ref{assumption:1}  in  Theorem~\ref{Thm:FE improved uniform}  is to obtain the estimate \eqref{estimate_b} on the whole domain $\Omega$. In fact, without  this assumption we can only establish \eqref{estimate_b} on a subdomain  of $\Omega$. In the following steps, we summarise the main modifications that apply to the theorem and its proof if we drop this assumption.

\noindent
\textit{Step~1}. We consider the estimate from \cite[Theorem~5.1]{schatz1977interior}, that is, 
\begin{align}
\|y - y_h\|_{L^\infty(\Omega_b)} \leq c  \Big( |\ln h| \inf_{\chi \in X_{h0} } \|y - \chi\|_{L^\infty(\Omega_a) } + \|y - y_h\|_{L^2(\Omega_a)} \Big) \label{estimate_c}
\end{align}
for some $\Omega_b \subset \subset \Omega_a \subset \subset \Omega$, where $y$ and $y_h$ are as defined in Lemma~\ref{Thm:ritz projection}.  We emphasise that \eqref{estimate_c} holds without   requiring Assumption~\ref{assumption:1}.

\noindent
\textit{Step~2}. We establish the inequality \eqref{inequality:a0} on $\Omega_b$ instead of $\Omega$. Then, we estimate the term (I) there using \eqref{estimate_c} instead of \eqref{estimate_a}. The rest of the modifications in the proof are obvious.

\noindent
\textit{Step~3}. We may now drop Assumption~\ref{assumption:1} from the hypothesis of Theorem~\ref{Thm:FE improved uniform} and replace \eqref{estimate_b} by
\begin{align*}
\| y -  y_h\|_{L^\infty(\Omega_b)} \leq  c |\ln h| h^{2-\frac{2}{p}} \big (\| u\|_{L^p(\Omega)} +1\big).
\end{align*}
\end{remark}
\end{comment}

The variational discretization (see \cite{hinze2005variational}) of Problem~$(\mathbb{P})$ reads:
\[
(\mathbb{P}_h) \quad
\begin{array}{l}
 \min_{u \in U_{ad}} J_h(u):=\frac{1}{2} \Vert y_h-y_0 \Vert_{L^2(\Omega)}^2  + \dfrac{\alpha}{2} \Vert u \Vert_{L^2(\Omega)} ^2 \\
\mbox{ subject to } y_h = \mathcal G_h(u), \,  (y_h(x_j))_{x_j \in \mathcal N_h} \in Y^h_{ad}, 
\end{array}
\]
where we define
\[
Y^h_{ad}:=\{ (z_j)_{x_j \in \mathcal N_h} \, | \, y_a(x_j) \leq z_j \leq y_b(x_j), x_j \in \mathcal N_h \},
\]
with the set of nodes
\begin{align*}
\mathcal N_h := \lbrace x_j \, | \, x_j \mbox{ is a vertex of } T \in \mathcal T_h, \mbox{ where } T \cap K \neq \emptyset \rbrace.
\end{align*}
We remark that $y_a(x_j) < y_b(x_j), x_j \in \mathcal N_h$ provided that $h$ is small enough. This follows from the fact that $\mbox{dist}(x_j,K) \leq h, x_j \in \mathcal N_h$  and $y_a,y_b$ are continuous functions  with  $y_a(x) <y_b(x), x \in K$.

In an analogous way to that of problem $(\mathbb{P})$, one can show that $(\mathbb{P}_h)$ admits at least one solution, denoted by $\bar u_h$, provided that a feasible point exists. In practice one
calculates candidates for solutions of $(\mathbb{P}_h)$ by solving the system of necessary first order conditions which reads: find $\bar u_h \in U_{ad}, \bar y_h \in X_{h0},
 \bar{p}_h \in X_{h0}$ and
$\bar{\mu}_j \in \mathbb{R}, x_j \in \mathcal N_h$ such that 
\begin{align}
&\int_{\Omega} \nabla \bar y_h \cdot \nabla v_h + \phi(\bar y_h) v_h \, dx  = \int_{\Omega} \bar u_h v_h \, dx  \quad \forall \, v_h \in X_{h0},  \qquad   (\bar y_h(x_j))_{x_j \in \mathcal N_h} \in Y^h_{ad}  \label{oc:state h} \\
&\int_\Omega \nabla \bar p_h \cdot \nabla v_h + \phi'(\bar y_h)\bar p_h v_h \, dx 
    \nonumber \\
& \enspace =\int_\Omega (\bar y_h -y_0) v_h \, dx  + \sum_{x_j \in \mathcal N_h} \bar \mu_j v_h(x_j)  \qquad \forall \, v_h \in X_{h0}, \label{oc:adjoint h}
 \\
& \int_\Omega (\bar p_h + \alpha \bar u_h)(u-\bar u_h) \, dx  \geq  0 \qquad \forall \, u \in U_{ad}, \label{oc:VI control h} \\
& \sum_{x_j \in \mathcal N_h} \bar \mu_j (z_j-\bar y_h(x_j))  \leq  0 \qquad \forall \, (z_j)_{x_j \in \mathcal N_h} \in Y^h_{ad}. \label{oc:VI measure h}
\end{align}
As in the continuous case, there exist multipliers $\bar p_h$ and $\bar{\mu}_j \in \mathbb{R}, x_j \in \mathcal N_h$ solving \eqref{oc:state h}--\eqref{oc:VI measure h}
provided that the local solution  $\bar u_h$ satisfies the linearized Slater condition, that is, there exist $u_0 \in U_{ad}$ and $\delta >0$ such that
\begin{align}
y_a(x_j)+\delta \leq \mathcal{G}_h(\bar u_h)(x_j)+ \mathcal{G}'_h(\bar u_h)(u_0 - \bar u_h)(x_j) \leq y_b(x_j)-\delta , \, x_j \in \mathcal N_h.  \label{slater condition h}
\end{align}
It will be convenient in the upcoming analysis to associate with the multipliers $(\bar{\mu}_j)_{x_j \in \mathcal N_h}$ from the system \eqref{oc:state h}--\eqref{oc:VI measure h}   the measure $\bar \mu_h \in \mathcal M(\Omega)$ defined by
\begin{align}
\label{discrete measure}
\bar \mu_h:= \sum_{x_j \in \mathcal N_h} \bar\mu_j \delta_{x_j},
\end{align}
where $\delta_{x_j}$ is the Dirac measure at $x_j$. We can easily deduce from  \eqref{oc:VI measure h} the following result about the support of the measure $\bar\mu_h$.

\begin{proposition}
\label{Thm:supp measure h}
Let $\bar \mu_h \in \mathcal{M}(\Omega)$ be the measure introduced in \eqref{discrete measure} satisfying \eqref{oc:VI measure h}. Then there holds
\begin{align*}
&\mbox{supp} (\bar \mu^b_h) \subset \{x_j \in \mathcal{N}_h: \bar y_h(x_j)=y_b(x_j) \},\\
&\mbox{supp} (\bar \mu^a_h) \subset \{x_j \in \mathcal{N}_h: \bar y_h(x_j)=y_a(x_j) \}.
\end{align*}
where $\bar \mu_h = \bar \mu^b_h - \bar \mu^a_h$  with  $\bar \mu^b_h,\bar \mu^a_h \geq 0$ is the Jordan decomposition of $\bar \mu_h$.  
\end{proposition}

Analogously to Theorem~\ref{Thm:global minima}, we have the next theorem about global solutions of problem $(\mathbb{P}_h)$. The proof can be found in \cite{ali2016global}.
\begin{theorem}
\label{Thm:global minima h}
Suppose that $\bar u_h \in U_{ad}$, $\bar y_h \in X_{h0}$, $\bar p_h \in X_{h0}$, $(\bar \mu_j)_{x_j \in \mathcal N_h}$ is a solution of \eqref{oc:state h}--\eqref{oc:VI measure h}.
If
\begin{align}
\label{18strich h}
\|\bar p_h\|_{L^q(\Omega)} \leq \eta(\alpha,r), 
\end{align}
then $\bar u_h$ is a global minimum for Problem~$(\mathbb{P}_h)$. If the inequality \eqref{18strich h} is strict, then $\bar u_h$ is the unique global minimum. 
\end{theorem}

\section{Error Analysis}
\label{sec:error}
Let $\{\mathcal{T}_h \}_{0<h \leq h_0}$ be a sequence of admissible triangulations of $\Omega$. We assume 
that the  sequence  $\{\mathcal{T}_h \}_{0<h \leq h_0}$ is quasi-uniform in the sense that each $T \in \mathcal{T}_h$ is contained in a ball of radius $\gamma^{-1}h$ and contains a ball of radius 
$\gamma h$ for some $\gamma >0$ independent of $h$. In addition we make the following assumption concerning the set $K$:

\begin{assumption}
\label{assumption:3}
For every $h>0$ there exists a set of triangles $ \mathbb{T}_h \subset \mathcal{T}_h$ such that
\[
K= \bigcup_{T \in \mathbb{T}_h} \bar T. 
\]
\end{assumption}

In what follows we consider a sequence $(\bar u_h, \bar y_h, \bar p_h, \bar \mu_h)_{0<h \leq h_1}$ of solutions of
\eqref{oc:state h}--\eqref{oc:VI measure h} satisfying
\begin{equation}
\label{inequality:strict ph}
\|\bar p_h\|_{L^q(\Omega)} \le \kappa \eta(\alpha,r) \qquad \mbox{ for } 0 < h \leq h_1
\end{equation}
for some $\kappa \in (0,1)$ that is independent of $h$. We immediately infer
from Theorem \ref{Thm:global minima h} that $\bar{u}_h$ is the unique  global minimum of $(\mathbb{P}_h)$ and we are interested
in the convergence properties of these solutions as $h \rightarrow 0$. It is shown in  \cite{ali2016global} (see Theorem 4.2 and
its proof) that there exist $\bar u \in U_{ad}$, $\bar p \in L^q(\Omega)$ and $\bar \mu \in \mathcal M(K)$ such that
\begin{displaymath}
\bar u_h \rightarrow \bar u \mbox{ in } L^2(\Omega), \quad  \bar p_h \rightharpoonup \bar p \mbox{ in } L^q(\Omega), \quad \bar \mu_h \rightharpoonup \bar \mu \mbox{ in }
\mathcal M(K)
\end{displaymath}
and  $(\bar u, \bar y = \mathcal G(\bar u), \bar p, \bar \mu)$ is a solution of \eqref{oc:state}--\eqref{oc:VI measure}. Since 
\begin{equation}
\Vert \bar p \Vert_{L^q(\Omega)} \leq \liminf_{h \rightarrow 0} \Vert \bar p_h \Vert_{L^q(\Omega)} \leq \kappa \eta(\alpha,r), \label{inequality:strict p}
\end{equation}
Theorem \ref{Thm:global minima}
implies that $\bar u$ is the unique global optimum of $(\mathbb{P})$. The aim in the remaining part of this paper is to 
prove error estimates for $\bar u_h - \bar u$ and the corresponding optimal states $\bar y_h - \bar y$. Our main results read: 

\begin{theorem}
\label{Thm:error estimate}
Suppose that \eqref{inequality:strict ph} holds  and let $\bar u_h, \, \bar{u}$ be the unique global minima of $(\mathbb{P}_h)$ and $(\mathbb{P})$
respectively. Then we have for any $1<s<2$ that
\begin{eqnarray}
\|\bar u_h -\bar u\|_{L^2(\Omega)} & \leq &   c_s \sqrt{|\ln h|} \, h^{\frac{3}{2}-\frac{1}{s}} \label{uestimate:optimal order} \\
\|\bar y_h-\bar y\|_{H^1(\Omega)} + \Vert \bar y_h - \bar y \Vert_{L^{\infty}(\Omega)} & \leq & c_s \sqrt{|\ln h|} \, h^{\frac{3}{2}-\frac{1}{s}}. \label{yestimate:optimal order}
\end{eqnarray}
\end{theorem}

\begin{remark}\label{PNR}
In \cite{neitzel2015finite} Pfefferer at al. for problems in two and three dimensions present a similar error estimate for discrete (local) solutions in the vicinity of a local solution which satisfies a quadratic growth condition. Assuming \eqref{inequality:strict ph} we here use different techniques to prove an error estimate for the unique global discrete solutions which converge to the unique global solution of our optimization problem.
\end{remark}

Before we start presenting the proof of this result we collect some results concerning the uniform boundedness of the discrete optimal control $\bar u_h$, its state $\bar y_h$ and the associated multipliers $\bar p_h$ and $\bar \mu_h$. 

\begin{lemma}
\label{Thm:uymu bounds}
Let $\bar u_h \in U_{ad}$, $\bar y_h$, $\bar p_h \in X_{h0}$ and $(\bar \mu_j)_{x_j \in \mathcal N_h}$ be a solution of \eqref{oc:state h}--\eqref{oc:VI measure h} satisfying
\begin{align*}
\|\bar p_h\|_{L^q(\Omega)} \leq \eta(\alpha,r),  \quad 0<h \leq h_0.
\end{align*}
Then there exists a constant $C>0$, which is independent of $h$, such that
\begin{equation}  \label{bounds}
\|\bar u_h\|_{L^2(\Omega)}, \, \|\bar y_h\|_{H^1(\Omega)}, \, \|\bar y_h\|_{L^\infty(\Omega)}\,  , \, \|\bar \mu_h\|_{\mathcal{M}(K)} \leq C. 
\end{equation}
\end{lemma}
\begin{proof}
The uniform boundedness of $\|\bar u_h\|_{L^2(\Omega)}, \, \|\bar y_h\|_{H^1(\Omega)} , \, \|\bar \mu_h\|_{\mathcal{M}(\tilde K)}$ is shown in \cite[Lemma~4.1]{ali2016global} while the one of $\|\bar y_h\|_{L^\infty(\Omega)}$ is a consequence of the uniform convergence \cite[(4.16)]{ali2016global}.
\end{proof}

Next, let us  introduce the  auxiliary functions $\tilde{y}^h \in H^2(\Omega)\cap H^1_0(\Omega)$, $\tilde{y}_h \in  X_{h0}$, $\tilde{p}_h \in  X_{h0}$ as the solutions of 
\begin{align}
&\int_\Omega \nabla \tilde y^h   \cdot  \nabla v +  \phi( \tilde y^h) v \, dx  =  \int_\Omega \bar u_h v \,dx \qquad \forall \, v \in H^1_0(\Omega), \label{equation:b0} \\
&\int_\Omega \nabla \tilde y_h   \cdot  \nabla v_h +  \phi( \tilde y_h) v_h \, dx = \int_\Omega \bar u v_h \,dx \qquad \forall \, v_h \in X_{h0}, \label{equation:b1}  \\
&\int_\Omega \nabla \tilde p_h \cdot \nabla v_h + \phi'(\bar y)\tilde p_h v_h \, dx 
  \nonumber \\
& \enspace = \int_\Omega (\bar y -y_0) v_h \, dx  + \int_K v_h \, d\bar \mu  \qquad \forall \, v_h \in X_{h0}.
\label{equation:b2}  
\end{align}

\begin{lemma} Let $\tilde y^h, \, \tilde y_h$ and $\tilde p_h$ be as above and $\Omega_0$ an open set such that $\overline{\Omega}_0  \subset \Omega$ and $K \subset \Omega_0$.
Then we have
\begin{align}
\label{estimate:b0}
\|\bar y_h - \tilde y^h\|_{L^2(\Omega)} + h \, \|\bar y_h - \tilde y^h\|_{L^{\infty}(\Omega)}  &\leq c h^2 \big( \|\bar u_h\|_{L^2(\Omega)} +1 \big),\\
\label{estimate:b1a}
\|\tilde y_h - \bar y\|_{L^2(\Omega)} + h \,  \|\tilde y_h - \bar y\|_{L^{\infty}(\Omega)}&\leq c h^2 \big ( \|\bar u\|_{L^2(\Omega)} +1 \big), \\
\label{estimate:b1b}
\|\tilde y_h - \bar y\|_{L^\infty(\Omega_0)} &\leq c |\ln h| h^{3-\frac{2}{s}} \big (\| \bar u\|_{W^{1,s}(\Omega)} +1\big), \\
\label{estimate:b2} 
\|\tilde p_h - \bar p\|_{L^2(\Omega)} &\leq c h \big( \|\bar y - y_0\|_{L^2(\Omega)} + \|\bar \mu \|_{\mathcal{M}(K)} \big). 
\end{align}
\end{lemma}
\begin{proof} The estimates \eqref{estimate:b0} and \eqref{estimate:b1a} can be found as Theorems 1 and 2 in \cite{casas2002uniform}. On the other hand, \eqref{estimate:b1b} follows from  \cite[Theorem 3.5]{neitzel2015finite}. 
Finally, the estimate \eqref{estimate:b2} is a consequence of \cite[Theorem~3]{casas1985}.
\end{proof}

\begin{comment}
\noindent
Next, we make an assumption on the gradient of the optimal state $\nabla \bar y$ at the set of points in $K$ where the state constraints are active. This assumption shall be mentioned explicitly wherever it is needed.
\begin{assumption}
\label{assumption:2} 
For the optimal state $\bar y$ there holds
\begin{align*}
\nabla \bar y(x)&=\nabla y_b(x) \qquad \forall \, x \in  K:\bar y(x)=y_b(x),\\
\nabla \bar y(x)&=\nabla y_a(x) \qquad \forall \, x \in K:\bar y(x)=y_a(x).
\end{align*}
\end{assumption}
\noindent
Notice that Assumption~\ref{assumption:2} is satisfied necessarily at the state constraints active points that belong to the interior of $K$, denoted by $\mbox{int}\, K$.  To see this, consider the function $f_b:K \to \mathbb{R}$ defined by $f_b:=\bar y-y_b$. If $f_b(x^\ast)=0$ for some $x^\ast \in K$, then $x^\ast$ is a global maximum for the function $f_b$ since $f_b \leq 0$ in $K$. Consequently,  the first order necessary optimality condition at $x^\ast$ reads  $\nabla f_b(x^\ast)=0$ provided that  $x^\ast \in \mbox{int}\, K$. On the other hand, if $x^\ast \in \partial K$, then $\nabla f_b(x^\ast) \not=0$ in general.

Since the optimal state $\bar y$ is usually not known a priori,  the previous assumption might be practically restrictive. For this reason, we suggest the next assumption which is relatively easier to be fulfilled while it leads to the same convergence rate for the discrete optimal controls that would be obtained under Assumption~\ref{assumption:2}.
\end{comment}

\noindent
{\bf Proof of Theorem \ref{Thm:error estimate}:}
Testing \eqref{oc:VI control} with $\bar u_h$ and \eqref{oc:VI control h} with $\bar u$ and adding the resulting inequalities gives
\begin{align*}
 \int_\Omega (\bar p_h-\bar p)(\bar u- \bar u_h ) - \alpha (\bar u_h-\bar u)^2 \, dx  \geq 0
\end{align*}
from which we obtain 
\begin{align} 
\alpha \| \bar u- \bar u_h \|^2_{L^2(\Omega)} &\leq  \int_\Omega (\bar u- \bar u_h ) (\bar p_h-\bar p)\, dx \nonumber\\
& = \underbrace{\int_\Omega (\tilde p_h-\bar p)(\bar u- \bar u_h ) \, dx}_{S_1} + \int_\Omega (\bar p_h-\tilde p_h)(\bar u- \bar u_h ) \, dx. \label{inequality:b0}
\end{align}
We see that from \eqref{oc:state h} and \eqref{equation:b1} with the choice $v_h=\bar p_h -\tilde p_h$ that
\begin{align*}
& \int_\Omega (\bar p_h-\tilde p_h)(\bar u- \bar u_h ) \, dx =   \int_\Omega [ \phi(\tilde y_h)-\phi(\bar y_h) ]  (\bar p_h -\tilde p_h) \, dx \\
& \quad  + \int_\Omega \nabla(\tilde y_h -\bar y_h)\cdot\nabla(\bar p_h -\tilde p_h) \,dx \\
&\quad = \underbrace{ \int_\Omega (\bar y_h-\bar y)(\tilde y_h - \bar y_h ) \, dx }_\text{$S_{2}$} + \underbrace{ \int_{K} (\tilde y_h- \bar y_h)\, d\bar\mu_h - \int_K (\tilde y_h - \bar y_h) \, d \bar \mu }_\text{$S_{3}$}  \\\
&\quad + \underbrace{ \int_\Omega [ \phi(\tilde y_h)-\phi(\bar y_h) ]  (\bar p_h -\tilde p_h) \, dx  - \int_\Omega [\phi'(\bar y_h)\bar p_h - \phi'(\bar y) \tilde p_h](\tilde y_h -\bar y_h) \, dx }_\text{$S_{4}$}, 
\end{align*}
where we utilized \eqref{oc:adjoint h} and \eqref{equation:b2} with the test function $v_h=\tilde y_h -\bar y_h$ to rewrite the term containing the gradients in the first equality. Consequently, adding the terms $S_{2}$, $S_{3}$, $S_4$ to $S_1$ in \eqref{inequality:b0}  gives
\begin{align}
\label{inequality:b1}
\alpha \| \bar u- \bar u_h \|^2_{L^2(\Omega)} \leq  \sum_{i=1}^4 S_i. 
\end{align}

\noindent
Young's inequality together with  \eqref{estimate:b2} implies that 
\begin{align*}
S_1 &\leq \| \tilde p_h-\bar p \|_{L^2(\Omega)} \|\bar u- \bar u_h \|_{L^2(\Omega)} 
 \leq \frac{1}{2\alpha \epsilon} \| \tilde p_h-\bar p \|^2_{L^2(\Omega)} + \frac{\alpha \epsilon}{2} \|\bar u- \bar u_h \|^2_{L^2(\Omega)} \\
& \leq  \frac{c}{\alpha \epsilon} h^2  \big( \|\bar y - y_0\|^2_{L^2(\Omega)} + \|\bar \mu \|^2_{\mathcal{M}(K)} \big) + \frac{\alpha \epsilon}{2} \|\bar u- \bar u_h \|^2_{L^2(\Omega)} \\
& =  \frac{\alpha \epsilon}{2} \|\bar u- \bar u_h \|^2_{L^2(\Omega)}  + c_{\epsilon} h^2.
\end{align*}
In a similar way we deduce with the help of \eqref{estimate:b1a} 
\begin{align*}
S_2 &=- \| \bar y_h-\bar y \|^2_{L^2(\Omega)} +\int_\Omega (\bar y_h-\bar y)(\tilde y_h - \bar y ) \, dx \\
&\leq - \| \bar y_h-\bar y \|^2_{L^2(\Omega)} +\| \bar y_h-\bar y \|_{L^2(\Omega)} \| \tilde y_h-\bar y \|_{L^2(\Omega)}\\
& \leq - \| \bar y_h-\bar y \|^2_{L^2(\Omega)}+ \frac{\epsilon}{2}\| \bar y_h-\bar y \|^2_{L^2(\Omega)} +\frac{1}{2\epsilon} \| \tilde y_h-\bar y \|^2_{L^2(\Omega)} \\
&  \leq ( \frac{\epsilon}{2} - 1)\| \bar y_h-\bar y \|^2_{L^2(\Omega)} +\frac{c}{\epsilon} h^4 \big ( \|\bar u\|^2_{L^2(\Omega)} +1 \big) \\
&  \leq  ( \frac{\epsilon}{2} - 1)\| \bar y_h-\bar y \|^2_{L^2(\Omega)} + c_{\epsilon} h^4.
\end{align*}
Let us next consider the first integral in $S_3$. Using $\bar\mu_h=\bar\mu^b_h-\bar\mu^a_h$, Proposition \ref{Thm:supp measure h}, the fact that $y_a \leq \bar y \leq y_b$ on $K$, Lemma~\ref{Thm:uymu bounds} and \eqref{estimate:b1b} we have
\begin{eqnarray}
\lefteqn{
\int_{K}  (\tilde y_h -\bar y_h) \, d\bar\mu_h = \int_{K}  (\tilde y_h-y_b) \, d\bar{\mu}_h^b -  \int_{K}  (\tilde y_h-y_a) \, d\bar{\mu}_h^a }  \nonumber  \\
& \leq &  \int_{K}  (\tilde y_h-\bar y) \, d\bar{\mu}_h^b +  \int_{K}  (\bar y-\tilde y_h) \, d\bar{\mu}_h^a 
\leq \|\tilde y_h - \bar y\|_{L^\infty(\Omega_0)} \|\bar{\mu}_h\|_{\mathcal{M}(K)} \nonumber \\ 
& \leq & c |\ln h| h^{3-\frac{2}{s}} \big (\| \bar u\|_{W^{1,s}(\Omega)} +1\big). \label{first integral}
\end{eqnarray}
\noindent
To estimate the second integral in $S_3$ we use Proposition~\ref{Thm:supp measure}, the fact that $ I_h y_a \leq \bar y_h \leq I_h y_b$ in $K$, a well--known interpolation  
estimate and \eqref{estimate:b1b} to obtain
\begin{align}
& \int_K  (\bar y_h - \tilde y_h ) \, d\bar\mu  = \int_K  (\bar y_h - \tilde y_h ) \, d\bar\mu_b - \int_K  (\bar y_h - \tilde y_h ) \, d\bar\mu_a \nonumber\\
& \quad\leq  \int_K  (I_h y_b -y_b) \, d\bar\mu_b + \int_K ( \bar y - \tilde y_h ) \, d\bar\mu_b + \int_K  ( \tilde y_h - \bar y ) \, d\bar\mu_a   +   \int_K  ( y_a - I_h y_a ) \, d\bar\mu_a  \nonumber\\
& \quad\leq  \|\bar \mu\|_{\mathcal{M}(K)} \Big( \| \tilde y_h - \bar y \|_{L^\infty(K)} + \| y_a - I_h y_a \|_{L^\infty(K)} + \| y_b - I_h y_b \|_{L^\infty(K)} \Big ) \nonumber\\
&\quad \leq c \|\bar \mu\|_{\mathcal{M}(K)} \Big ( |\ln h| h^{3-\frac{2}{s}} (\|\bar u\|_{W^{1,s}(\Omega)} +1 ) + c h^2 \bigl( \|y_a\|_{W^{2,\infty}(\Omega)} +  \|y_b\|_{W^{2,\infty}(\Omega)} \bigr) \Big) \nonumber\\
& \quad\leq c |\ln h| h^{3-\frac{2}{s}} \|\bar \mu\|_{\mathcal{M}(K)} \bigl(  \|\bar u\|_{W^{1,s}(\Omega)}  +  1 \bigr).  \label{second integral}
\end{align}
Combining \eqref{first integral} and \eqref{second integral}  yields 
\begin{equation*}
S_3 \leq c  |\ln h| h^{3-\frac{2}{s}} \bigl( \|\bar u\|_{W^{1,s}(\Omega)}  +  1 \bigr).
\end{equation*}
\noindent
Let us next turn  $S_4$, which we rewrite as 
\begin{align*}
& S_4 =\underbrace{ \int_\Omega [\phi(\tilde y_h)-\phi(\bar y_h)+\phi'(\bar y_h)(\bar y_h-\tilde y_h)] \bar p_h \,dx}_{S_{4.1}} + \underbrace{ \int_\Omega [\phi(\tilde y^h)-\phi(\bar y)+\phi'(\bar y)(\bar y-\tilde y^h)] \bar p \,dx }_{S_{4.2}} \\
& + \underbrace{ \int_\Omega [\phi(\tilde y^h)-\phi(\bar y)+\phi'(\bar y)(\bar y-\tilde y^h)](\tilde p_h -\bar p) \,dx }_{S_{4.3}} + \underbrace{ \int_\Omega [\phi(\bar y)-\phi(\tilde y_h)+\phi'(\bar y)(\tilde y_h-\bar y)]\tilde p_h \,dx }_{S_{4.4}} \\
& + \underbrace{ \int_\Omega [\phi(\bar y_h)-\phi(\tilde y^h) +\phi'(\bar y)(\tilde y^h -\bar y_h )]\tilde p_h \,dx}_{S_{4.5}}.
\end{align*}

\noindent
In order to estimate $S_{4.1}$ we first observe that $S_{4.1}=R_h(u_h)$ for the choice $y_h=\tilde y_h$, where $R_h(u_h)$ is defined at the bottom of page 266 in 
\cite{ali2016global}. Retracing the steps in \cite{ali2016global} leading to (3.11) we infer that
\begin{equation}  \label{side calculation: 3}
\displaystyle
|  S_{4.1}  |   \leq 2 \alpha^{-\frac{\rho}{2}} L_r C^{\frac{2r-2}{r}}_q d_r e_r  \|\bar p_h\|_{L^q(\Omega)} 
 \Big(  \frac{1}{2} \| \tilde y_h- \bar y_h \|_{L^2(\Omega)}^2 +   \frac{\alpha}{2} \| u_h- \bar u_h \|_{L^2(\Omega)}^2   \Big),   
\end{equation}
where $q$ and $\rho$ are defined immediately after \eqref{eta def}, while
\[ 
L_r = M \bigl( \frac{r-1}{2r-1} \bigr)^{(r-1)/r}, \quad 
d_r = q^{-1/q}  r^{-1/r} \rho^{-\rho}, 
e_r = \bigl( 1 - \frac{\rho}{2} \bigr)^{1- \frac{\rho}{2}} \bigl( \frac{\rho}{2} \bigr)^{\frac{\rho}{2}}.
\]
In view of the definition of $\eta(\alpha,r)$ and \eqref{inequality:strict ph} this implies 
\begin{eqnarray*}
| S_{4.1} |& \leq & \eta(\alpha,r)^{-1} \|\bar p_h\|_{L^q(\Omega)} \Big(\dfrac{1}{2}\|\tilde y_h-\bar y_h\|^2_{L^2(\Omega)} + \dfrac{\alpha}{2}\|\bar u -\bar u_h\|^2_{L^2(\Omega)}  \Big) \\
& \leq & \kappa \Big(\dfrac{1}{2}\|\tilde y_h-\bar y_h\|^2_{L^2(\Omega)} + \dfrac{\alpha}{2}\|\bar u -\bar u_h\|^2_{L^2(\Omega)}  \Big).
\end{eqnarray*}
Since
\begin{eqnarray*}
\Vert \tilde y_h - \bar y_h \Vert^2_{L^2(\Omega)} & \leq & (1+ \epsilon) \|\bar y_h-\bar y\|^2_{L^2(\Omega)} + c_{\epsilon} \|\tilde y_h-\bar y\|^2_{L^2(\Omega)} \\
& \leq & (1+ \epsilon) \|\bar y_h-\bar y\|^2_{L^2(\Omega)} + c_{\epsilon} h^4
\end{eqnarray*}
by \eqref{estimate:b1a}, we finally  obtain
\begin{displaymath}
| S_{4.1} | \leq \kappa
 \Big(\frac{1+\epsilon}{2}\|\bar y_h-\bar y\|^2_{L^2(\Omega)} + \frac{\alpha}{2}\|\bar u -\bar u_h\|^2_{L^2(\Omega)}  \Big) + c_{\epsilon} h^4.
\end{displaymath}
Using \eqref{inequality:strict p}, \eqref{estimate:b0} and \eqref{bounds}, we derive in a similar way   
\begin{align*}
| S_{4.2} | & \leq \eta(\alpha,r)^{-1} \|\bar p\|_{L^q(\Omega)} \Big(\dfrac{1}{2}\|\tilde y^h-\bar y\|^2_{L^2(\Omega)} + \dfrac{\alpha}{2}\|\bar u_h -\bar u\|^2_{L^2(\Omega)}  \Big) \\
& \leq \kappa
 \Big(\frac{1+\epsilon}{2}\|\bar y_h-\bar y\|^2_{L^2(\Omega)} + \dfrac{\alpha}{2}\|\bar u_h -\bar u\|^2_{L^2(\Omega)}  \Big) +  c_{\epsilon} \| \tilde y^h - \bar y_h \|^2_{L^2(\Omega)}\\
 & \leq  \kappa
 \Big(\frac{1+\epsilon}{2}\|\bar y_h-\bar y\|^2_{L^2(\Omega)} + \dfrac{\alpha}{2}\|\bar u_h -\bar u\|^2_{L^2(\Omega)}  \Big) + c_{\epsilon} h^4.
\end{align*}
Since $\phi \in C^2$ and $\Vert \tilde y^h \Vert_{L^{\infty}(\Omega)}$ is uniformly bounded in $h$ (in view of \eqref{estimate:b0} and \eqref{bounds}) we infer with the help of Lemma \ref{Thm:lipschitz map} and  \eqref{estimate:b2} 
\begin{align*}
| S_{4.3} |  & \leq c \Vert \tilde y^h - \bar y \Vert_{L^{\infty}(\Omega)} \,  \|\tilde y^h -\bar y\|_{L^2(\Omega)}\|\tilde p_h -\bar p\|_{L^2(\Omega)} \\
& \leq  c h  \|\bar u_h -\bar u\|_{L^2(\Omega)} \big( \|\bar y - y_0\|_{L^2(\Omega)} + \|\bar \mu \|_{\mathcal{M}(K)} \big) \\
& \leq  \frac{\alpha \epsilon}{2}\|\bar u_h -\bar u\|^2_{L^2(\Omega)} + c_{\epsilon} h^2.
\end{align*}

\noindent
In a similar way we obtain  using \eqref{estimate:b1a} and \eqref{estimate:b0}
\begin{align*}
| S_{4.4} | + | S_{4.5} |  & \leq   c \bigl(   \|\tilde y_h -\bar y\|_{L^2(\Omega)} + \|\tilde y^h -\bar y_h\|_{L^2(\Omega)} \bigr) \|\tilde p_h\|_{L^2(\Omega)} \\
 & \leq c h^2 \bigl( \|\bar u\|_{L^2(\Omega)} + \|\bar u_h\|_{L^2(\Omega)} +1 \bigr) \|\tilde p_h\|_{L^2(\Omega)} \\
 & \leq c h^2,
\end{align*}
where we note that  $\|\tilde p_h\|_{L^2(\Omega)}$ is uniformly bounded for sufficiently small $h$ in view of \eqref{estimate:b2}. 
\noindent
Collecting the estimates for $S_{4.1},\ldots,S_{4.5}$, we conclude that  $S_{4}$ can be bounded by  
\begin{align*}
S_{4} \leq 
  \kappa (1+\epsilon)\|\bar y_h-\bar y\|^2_{L^2(\Omega)} + (\kappa \alpha  + \frac{\alpha \epsilon}{2} ) \|\bar u_h -\bar u\|^2_{L^2(\Omega)} + c_{\epsilon} h^2.
\end{align*}
Inserting the estimates of the terms $S_1, \ldots, S_4$ into \eqref{inequality:b1} yields
\begin{align*}
\alpha  \|\bar u_h -\bar u\|^2_{L^2(\Omega)} & \leq \big( \kappa (1+\epsilon) + (\frac{\epsilon}{2}-1)   \big) \|\bar y_h-\bar y\|^2_{L^2(\Omega)}  + \alpha (\kappa + \epsilon ) \|\bar u_h -\bar u\|^2_{L^2(\Omega)} \\
& \quad + c_{\epsilon} |\ln h| h^{3-\frac{2}{s}} \bigl( \Vert \bar u \Vert_{W^{1,s}(\Omega)} + 1 \bigr). 
\end{align*}
Since $\kappa <1$,  choosing $\epsilon >0$ to be small enough in the above expression yields the existence  of $c >0$ independent of $h$ such that
\begin{equation}  \label{prelimbound}
 \|\bar u_h -\bar u\|^2_{L^2(\Omega)} +  \|\bar y_h-\bar y\|^2_{L^2(\Omega)} \leq  c |\ln h| h^{3-\frac{2}{s}}  \bigl( \Vert \bar u \Vert_{W^{1,s}(\Omega)} + 1 \bigr).
\end{equation}
Let us next establish an upper bound for $\| \nabla( \bar y_h-\bar y) \|_{L^2(\Omega)}$. To this end we introduce $R_h \bar y$ as the Ritz projection of $\bar y$, i.e
\begin{displaymath}
\int_{\Omega} \nabla R_h \bar y \cdot \nabla w_h \, dx = \int_{\Omega} \nabla \bar y \cdot \nabla w_h \, dx \qquad \forall  \, w_h \in X_{h0}.
\end{displaymath}
Let us first derive an upper bound on  $\| \nabla (\bar y_h-R_h \bar y)\|_{L^2(\Omega)}$. To begin, from the definition of $R_h \bar y$ and the weak formulation of $\bar y$ we have
\begin{align*}
\int_{\Omega} \nabla R_h \bar y \cdot \nabla w_h \, dx = \int_{\Omega} \nabla  \bar y \cdot \nabla w_h \, dx = \int_{\Omega}  \bar u   w_h \, dx - \int_{\Omega}  \phi(\bar y)   w_h \, dx \quad \forall \, w_h \in X_{h0}.
\end{align*}
If we combine this relation with \eqref{oc:state h} we obtain for all $w_h \in X_{h0}$ that
\begin{align*}
\int_{\Omega} \nabla (R_h \bar y-\bar y_h) \cdot \nabla w_h \, dx = \int_{\Omega}  (\bar u- \bar u_h )   w_h \, dx + \int_{\Omega} [ \phi(\bar y_h)-\phi(\bar y) ]   w_h \, dx.
\end{align*}
Using $w_h = R_h \bar y-\bar y_h$ in the previous variational equation and observing that  $\|\bar y_h\|_{L^\infty(\Omega)}$ is uniformly bounded in $h$ we deduce that
\begin{eqnarray*}
\lefteqn{ \int_{\Omega} |\nabla (R_h \bar y-\bar y_h)|^2  \, dx } \\
& \leq &  \big( \|\bar u- \bar u_h\|_{L^2(\Omega)}  + \|\phi(\bar y_h)-\phi(\bar y)\|_{L^2(\Omega)} \big) \|R_h \bar y-\bar y_h\|_{L^2(\Omega)} \\
& \leq & c \big( \|\bar u- \bar u_h\|_{L^2(\Omega)}  + \|\bar y_h-\bar y\|_{L^2(\Omega)} \big) \|\nabla(R_h \bar y-\bar y_h)\|_{L^2(\Omega)}
\end{eqnarray*}
by Poincar\'{e}'s inequality. Thus,
\begin{align}
\label{inequality_b}
\| \nabla (\bar y_h-R_h \bar y)\|_{L^2(\Omega)} \leq c \big( \|\bar u- \bar u_h\|_{L^2(\Omega)}  +  \|\bar y_h-\bar y\|_{L^2(\Omega)} \big),
\end{align}
which together with a standard error bound for the Ritz projection and \eqref{prelimbound} implies
\begin{align*}
\| \nabla( \bar y_h-\bar y) \|_{L^2(\Omega)} & \leq \| \nabla( \bar y_h-R_h \bar y) \|_{L^2(\Omega)} + \| \nabla( R_h \bar y-\bar y)\|_{L^2(\Omega)} \nonumber\\
& \leq  c \bigl( \|\bar u- \bar u_h\|_{L^2(\Omega)}  +  \|\bar y_h-\bar y\|_{L^2(\Omega)} \bigr)   + c h \|\bar y\|_{H^2(\Omega)} \\
& \leq c |\ln h| h^{3-\frac{2}{s}}  \bigl( \Vert \bar u \Vert_{W^{1,s}(\Omega)} + 1 \bigr).
\end{align*}
It remains to prove the uniform estimate for  $\bar y_h - \bar y$. We obtain from \eqref{estimate:b0}, the continuous embedding $H^2(\Omega) \hookrightarrow C(\bar\Omega)$, Lemma~\ref{Thm:H2 lipschitz} and \eqref{prelimbound} that
\begin{eqnarray*}
\lefteqn{
\| \bar y_h -  \bar y\|_{L^\infty(\Omega)} \leq \| \bar y_h - \tilde y^h\|_{L^\infty(\Omega)} + \|\ \tilde y^h - \bar y \|_{L^\infty(\Omega)}  } \\
& \leq &  c h \big( \|\bar u_h\|_{L^2(\Omega)}+1 \big) + c \| \tilde y^h - \bar y\|_{H^2(\Omega)} 
\leq c h + c \|\bar u_h-\bar u\|_{L^2(\Omega)} \\
& \leq &  c |\ln h| h^{3-\frac{2}{s}}  \bigl( \Vert \bar u \Vert_{W^{1,s}(\Omega)} + 1 \bigr).
\end{eqnarray*}
This completes the proof of Theorem \ref{Thm:error estimate}.

\begin{remark}
The choice $K=\bar \Omega$ is in fact  allowed for Problem~$(\mathbb{P})$ provided that the bounds $y_a$, $y_b \in C(\bar \Omega)$ 
satisfy in addition to $y_a <y_b $ in $\bar\Omega$ the compatibility condition $y_a <0<y_b$ on $\partial \Omega$. In this case the set
$\mathcal{N}_h$, which appears in the discrete optimal control problem, should be defined as
\[
\mathcal{N}_h := \{x_j | x_j \mbox{ is  a vertex of } T \in \mathcal{T}_h,  x_j \notin \partial \Omega \}.
\]
We claim that the assertion of  Theorem~\ref{Thm:error estimate} remains valid in this setting. To see this, we note that the only change in the proof concerns the term $S_3$ which now reads
\begin{align*}
\int_{\bar \Omega}  (\tilde y_h-\bar y) \, d\bar{\mu}_h   - \int_{\bar \Omega}  (\tilde y_h-\bar y) \, d\bar{\mu}.   
\end{align*}
However, using the fact that $y_a$, $y_b \in C(\bar \Omega)$,  $y_a <y_b $ in $\bar\Omega$ and $y_a <0<y_b$ on $\partial \Omega$, it can be shown that there exists $\Omega_0 \subset \subset \Omega$ such that $\mbox{supp} (\bar\mu) \subset \Omega_0$ and  $\mbox{supp} (\bar\mu_h) \subset \Omega_0$ for $h$  small enough, see \cite[Corollary~5.4]{casas2014new}. We may then use again \eqref{estimate:b1b} and argue in the same way as before. 
\end{remark}

\section{Numerical Example}
\label{sec:numerics}

We now examine numerically the error bounds established in Theorem~\ref{Thm:error estimate}. For this purpose, we consider the following example  taken from \cite[Section~7]{neitzel2015finite}, in which  
Problem $(\mathbb{P})$ is considered with the following choice for the data: $\Omega:=(0,1) \times (0,1)$,  $\phi(s)=s^3$,  $\alpha=10^{-2}$, $y_0 :=-1$, $U_{ad}=L^2(\Omega)$, $y_b=\infty$ and
\[
y_a(x):=-\frac{1}{2}+\frac{1}{2}\min ( x_1 + x_2, 1+ x_1 - x_2, 1- x_1 + x_2, 2-x_1 -x_2 ).
\]
It was shown in \cite{ali2016global} that this example admits a unique global solution. In fact, it is easy to see that  \eqref{assumption:D2phi} holds for  $\phi(s)=s^3$ with $r=2$ and $M=2\sqrt{3}$, hence $q=4$ in \eqref{eta def}. 
After applying the variational discretization, the numerical solution of the resulting discrete optimality system \eqref{oc:state h}--\eqref{oc:VI measure h} is obtained by the semismooth Newton method proposed in \cite{deckelnick2007finite} whose extension to semilinear elliptic control problems is straightforward. 
Consequently, the condition \eqref{18strich h} from Theorem~\ref{Thm:global minima h} now reads
\[
\|\bar p_h\|_{L^4(\Omega)} \leq  5^{-\frac{5}{8}} 3^{\frac{3}{8}} \sqrt{2} C^{-1}_4 \alpha^{\frac{3}{8}} =: \eta(\alpha),
\]
where $C_4  \approx 0.648027075$ is an upper bound for the constant in Gagliardo-Nirenberg inequality, precisely it is the bound $C_4^{(3)}$ from \cite[Theorem~7.3]{ali2016global}.  Figure~\ref{figure:unique} compares  the quantities  $\|\bar p_h\|_{L^4(\Omega)}$ and $\eta(\alpha)$ for several choices of $\alpha$, including $\alpha=10^{-2}$. It can be seen from this figure that  the previous condition is satisfied strictly which in turn implies that the considered 
example admits a unique global solution. 
The  global minimum of the considered example together with its state and the associated multipliers are presented graphically in Figure~\ref{figure: example y3 case(NeitzelExample)}. We see  that the state constraints are active at one point, namely $\tilde{x}:=(\frac{1}{2},\frac{1}{2})$, and the corresponding multiplier is approximately given by
\[
\bar \mu^a_h=0.3386 \, \delta_{\tilde{x}},
\]
where $\delta_{\tilde{x}}$ is a Dirac measure at $\tilde{x}$.  We can easily find a polygonal subdomain $K \subset \subset \Omega$ that contains the active point $\tilde x$  so that Assumption~\ref{assumption:3} holds. 
Consequently, we are expecting the bound $\sqrt{|\ln h|}  h^{\frac{3}{2}-\frac{1}{s}}$, or equivalently  $h^{1-\varepsilon}$ for arbitrarily small $\varepsilon >0$, for the computed errors according to  Theorem~\ref{Thm:error estimate}.

To deduce the convergence rates numerically,  we compute the experimental order of convergence (EOC) which is defined  as
\begin{align*}
\mbox{EOC}:= \dfrac{\log E(h_i) - \log E(h_{i-1})}{\log h_i - \log h_{i-1}},
\end{align*}
where $E$ is a given positive error functional and $h_{i-1}$, $h_{i}$ are two consecutive mesh sizes.  For our experiment, we consider the error functionals
\begin{align*}
E_{u_{L2}}(h_i) &:=\|\bar u_{ref}- \bar u_{h_i}\|_{L^2(\Omega)}, \\
E_{y_{H1}}(h_i) &:=\|\bar y_{ref}- \bar y_{h_i}\|_{H^1(\Omega)}, \\
E_{y_{L2}}(h_i) &:=\|\bar y_{ref}- \bar y_{h_i}\|_{L^2(\Omega)}, \\
E_{y_{L\infty}}(h_i) &:=\|\bar y_{ref}- \bar y_{h_i}\|_{L^{\infty}(\Omega)}, 
\end{align*}
and denote the corresponding experimental orders of convergence by EOC$_{u_{L2}}$, EOC$_{y_{H1}}$, EOC$_{y_{L2}}$ and EOC$_{y_{L\infty}}$, respectively. Furthermore, we consider the sequence of mesh sizes $h_i=2^{-i}\sqrt{2}$, for $i=1, \ldots, 9$. Since we don't have the exact solution  at hand,  we consider the numerical solution computed at mesh size $h_{10}=2^{-10}\sqrt{2}$ to be the reference solution, that is,  we define $\bar u_{ref}:= \bar u_{h_{10}}$ and $\bar y_{ref}:= \bar y_{h_{10}}$. 

Figure~\ref{F3} shows the values of our error functionals in dependence of $h$, and also illustrates the order of convergence. The computed values of the associated EOC are presented in Table~\ref{table:eoc}.

From the numerical findings we see that as the mesh size $h$ decreases the errors $E_{u_{L2}}(h)$ and $E_{y_{H1}}(h)$  behave like $O(h)$ which indicates that the convergence rate,  namely $O(h^{1-\varepsilon})$ for arbitrarily small $\varepsilon >0$, predicted in Theorem~\ref{Thm:error estimate} is optimal. On the other hand, for  $E_{y_{L2}}(h)$ and  $E_{y_{L\infty}}(h)$ we see the behaviour $O(h^2)$ and $O(h^{1.6})$, respectively, from which  we conclude that the error bounds for the discrete optimal state  in the spaces $L^2(\Omega)$ and $L^{\infty}(\Omega)$ which are deduced from the error bound of the discrete optimal control via the Lipschitz continuity of the control-to-state map are not sharp.

In fact, the $O(h^2)$ behaviour of $E_{y_{L2}}(h)$ could be explained in the light of the work \cite{neitzel2016wollner} where it was shown that for an elliptic control problem with finitely many pointwise inequality state constraints the error of the discrete optimal state in $L^2(\Omega)$ is of order $h^{4-d}$ up to logarithmic factor in $d=2$ or $d=3$ space dimensions  when the control problem is discretized by continuous, piecewise linear finite elements.

% \begin{table}[h!]
% \centering
% \caption{Errors for the optimal control and its state.}
% \label{table:error}
% \setlength{\tabcolsep}{12pt}
% \renewcommand{\arraystretch}{1.4}

% \begin{tabular}{l l l l l}
% \toprule
% $h/\sqrt{2}$ &	  $E_{u_{L2}}(h) $ &        $E_{y_{H1}}(h) $ &	  $E_{y_{L2}}(h) $ &        $E_{y_{L\infty}}(h) $  \\
% \midrule[1.5pt]

% $2^{-1}$ & 1.6151338799381 & 0.1881784634445 & 0.0305960611799 &  0.0465574255352 \\
% $2^{-2}$ & 0.7094890598326 & 0.1098931145143 & 0.0140288901847 & 0.0250260522632  \\
% $2^{-3}$ & 0.3114790933874 & 0.0616536701645 & 0.0050825086327 & 0.0120171679655  \\
% $2^{-4}$ & 0.1475243025114 & 0.0319521773619 & 0.0015547890304 & 0.0035520156534  \\
% $2^{-5}$ & 0.0723799608480 & 0.0161391228724 & 0.0004482149627 & 0.0011256351724  \\
% $2^{-6}$ & 0.0357734199802 & 0.0080807758796 & 0.0001259623551 & 0.0003940261492  \\
% $2^{-7}$ & 0.0174753747282 & 0.0040194139919 & 0.0000345959726 & 0.0001271705046  \\
% $2^{-8}$ & 0.0081450867872 & 0.0019615865181 & 0.0000090359657 & 0.0000390725086  \\
% $2^{-9}$ & 0.0032211298694 & 0.0008772740328 & 0.0000019527627 & 0.0000116353242  \\

% \bottomrule 
% \end{tabular}
% \end{table}

\begin{table}[h!]
\centering
\caption{EOC for the optimal control and its state.}
\label{table:eoc}
\setlength{\tabcolsep}{12pt}
\renewcommand{\arraystretch}{1.4}

\begin{tabular}{l l l l l}
\toprule
Levels & EOC$_{u_{L2}}$ &  EOC$_{y_{H1}}$ &  EOC$_{y_{L2}}$ &  EOC$_{y_{L\infty}}$ \\
\midrule[1.5pt]

1-2 & 1.186801 & 0.776001 & 1.124945 & 0.895581 \\
2-3 & 1.187645 & 0.833842 & 1.464788 & 1.058334 \\
3-4 & 1.078183 & 0.948273 & 1.708822 & 1.758387 \\
4-5 & 1.027290 & 0.985352 & 1.794456 & 1.657899 \\
5-6 & 1.016702 & 0.997996 & 1.831198 & 1.514376 \\
6-7 & 1.033565 & 1.007509 & 1.864317 & 1.631527 \\
7-8 & 1.101321 & 1.034964 & 1.936853 & 1.702538 \\
8-9 & 1.338363 & 1.160921 & 2.210162 & 1.747642 \\

\bottomrule 
\end{tabular}
\end{table}

\begin{figure}[p]
	\centering
	\includegraphics[trim = 38mm 80mm 30mm 70mm, clip, width=0.5\textwidth]{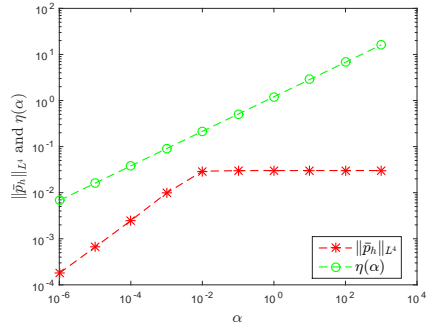}
	\caption{$\|\bar p_h\|_{L^4}$ and $\eta(\alpha)$ vs. $\alpha$.}
	\label{figure:unique}
\end{figure}

%Graphs/
\begin{figure}[p]
        \centering
        \begin{subfigure}[h!]{0.5\textwidth}
                  \includegraphics[trim = 40mm 80mm 30mm 70mm, clip, width=\textwidth]{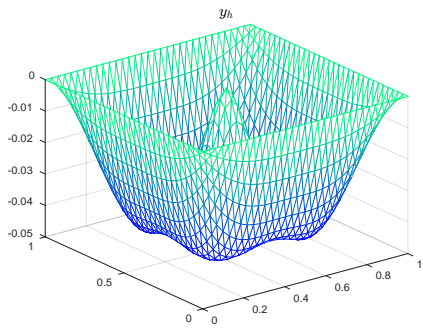}
                \caption{The optimal state $\bar y_h$.}
        \end{subfigure}~
        \begin{subfigure}[h!]{0.5\textwidth}
                  \includegraphics[trim = 40mm 80mm 30mm 70mm, clip, width=\textwidth]{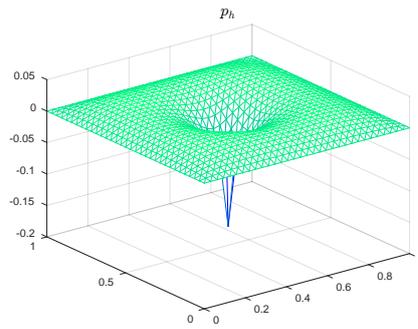}
                \caption{The adjoint state $\bar p_h$.}
        \end{subfigure}
        
        \begin{subfigure}[h!]{0.5\textwidth}
                  \includegraphics[trim = 40mm 80mm 30mm 70mm, clip, width=\textwidth]{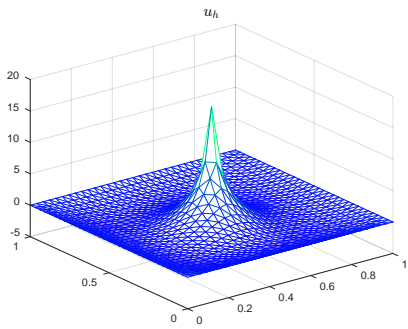}
                \caption{The optimal control $\bar u_h$.}
        \end{subfigure}~
        \begin{subfigure}[h!]{0.5\textwidth}
                  \includegraphics[trim = 40mm 80mm 30mm 70mm, clip, width=\textwidth]{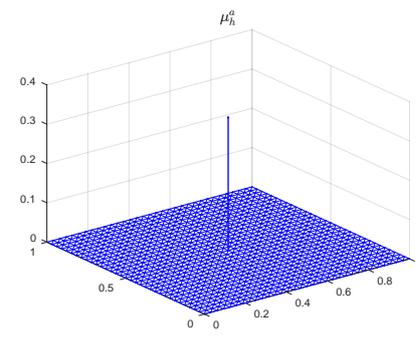}
                \caption{The multiplier $\bar \mu^a_h=0.3386 \, \delta_{\tilde{x}}$,  $\tilde{x}:=(\frac{1}{2},\frac{1}{2})$.}
        \end{subfigure}
        
        \caption{The unique global minimum together with its state and the associated multipliers.}
        \label{figure: example y3 case(NeitzelExample)}
\end{figure}

\begin{figure}[p]
        \centering
        \begin{subfigure}[h!]{0.90\textwidth}
                \includegraphics[trim = 40mm 85mm 40mm 90mm, clip,width=\textwidth]{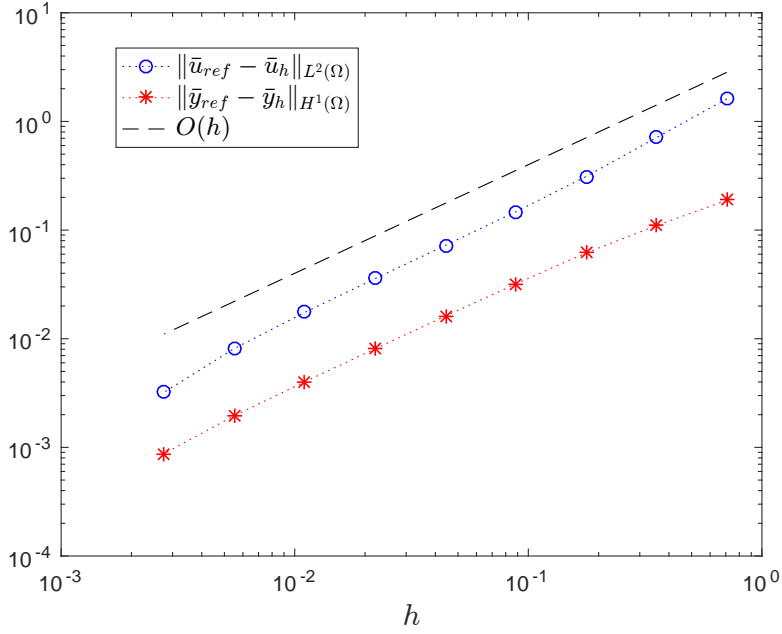}
                \caption{ $E_{u_{L2}}(h)$ and  $E_{y_{H1}}(h)$ v.s. $h$.}
        \end{subfigure}\\%
        \begin{subfigure}[h!]{0.90\textwidth}
               \includegraphics[trim = 40mm 85mm 40mm 85mm, clip,width=\textwidth]{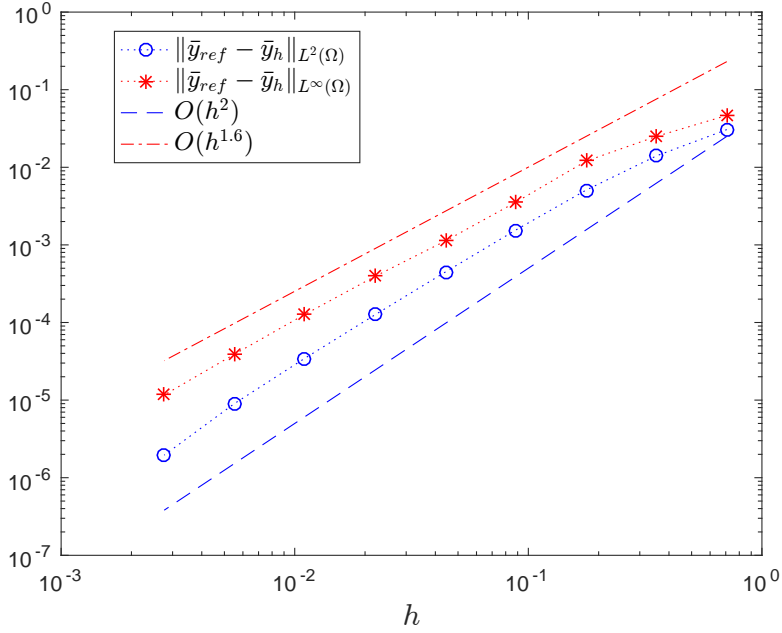}
                \caption{$E_{y_{L2}}(h)$ and  $E_{y_{L\infty}}(h)$ v.s. $h$.}
        \end{subfigure}%
        
        \caption{Errors for the optimal control and its state  versus the mesh size.}
        \label{F3}
\end{figure}

\bibliographystyle{plain}
\bibliography{references}
\end{document}